\documentclass[twoside,12pt]{amsart}

\usepackage{amsmath,latexsym,amssymb,mathptm, times,verbatim, enumerate,times}

\input amssym.def
\input amssym
\input xypic
\input xy

\xyoption{all}
\def \Z{\mathbb Z}
\def \init{\operatorname{init}}
\def \hn{\operatorname{hn}}
\def \HS{\operatorname{HS}}
\def \HN{\operatorname{HN}}

\def\grade{\operatorname{grade}}

\newcommand{\gf}[2]{\genfrac{}{}{0pt}{}{#1}{#2}}

\def\a{\alpha}

\newtheorem{theorem}{Theorem}[section]
\newtheorem{lemma}[theorem]{Lemma}
\newtheorem{corollary}[theorem]{Corollary}
\newtheorem{proposition}[theorem]{Proposition}

\newtheorem*{Main Theorem}{Main Theorem}
\newtheorem*{Main Corollary}{Main Corollary}

\newtheorem{def&dis}[theorem]{Definition and Discussion}
\theoremstyle{definition}
\newtheorem{definition}[theorem]{Definition}

\newtheorem*{Remark}{Remark}

\newtheorem{chunk}[theorem]{}

\newtheorem*{proof1}{Proof of Theorem \ref{Main-T}}
\newtheorem*{proof2}{Proof of Corollary \ref{Main-C}}

\numberwithin{equation}{theorem}

\begin{document}

\baselineskip=16pt

\title[The Hilbert series of the ring associated to an almost alternating matrix]{\bf The Hilbert series of the ring associated to an almost alternating matrix}
\date\today

\author[Andrew R. Kustin, Claudia Polini, and Bernd Ulrich]
{Andrew R. Kustin, Claudia Polini, and Bernd Ulrich}

\thanks{AMS 2010 {\em Mathematics Subject Classification}.
13H15,  
13C40,
13D02.  
}

\thanks{The first author was partially supported by  the Simons Foundation.
The second and third authors were partially supported by the NSF}


\address{Department of Mathematics, University of South Carolina,
Columbia, SC 29208} \email{kustin@math.sc.edu}

\address{Department of Mathematics, 
University of Notre Dame
Notre Dame, IN 46556} \email{cpolini@nd.edu}

\address{Department of Mathematics,
Purdue University,
West Lafayette, IN 47907}\email{ulrich@math.purdue.edu}

 \begin{abstract} We give an explicit formula for the Hilbert Series of an algebra defined by a linearly presented, standard graded, residual intersection of a grade three Gorenstein ideal.   \end{abstract}
 
\maketitle

\section{Introduction.}

Finding explicit formulas for  Hilbert series of residual  intersections is a matter of serious concern; see, for example, \cite{CEU,CEU13}.  The first paper shows that there should be formulas which express the Hilbert series of a scheme in terms of the Hilbert series of  its conormal modules. The second paper  exhibits the explicit formulas and has applications to the dimension of secant varieties. The Hilbert series of the present paper are completely explicit; there is no need to use the  Hilbert series of  conormal modules.

The notion of residual intersection (see \ref{ri}) was introduced by Artin and Nagata
\cite{AN};
it
has been improved and generalized by Huneke \cite{H83} and Huneke and
Ulrich
\cite{HU}. Residual intersections can be used to compute $j$-multiplicity \cite{xie}, the dimension of secant varieties \cite{CEU13},  complete intersection defect ideals  \cite{deFD}, Segre classes of subschemes of projective space \cite{EJP}, and Chern numbers of smooth varieties \cite{DEPS}.

  The residual intersections that we consider all arise from an almost alternating matrix.
If $n$ and $t$ are integers with $n$ positive and $t$ non-negative, then an $n\times (n+t)$ matrix $\rho$ is called {\it almost alternating} if the left-most $n$ columns of $\rho$ form an alternating matrix. We are interested in non-square almost alternating matrices. Let $R$ be a commutative Noetherian ring,  $n$ and $t$ be positive integers,  and  
 $\rho$ be an $n$ by  $(n+t)$  almost alternating matrix  with entries in $R$.
The matrix $\rho$ gives rise to an ideal   $J(\rho)$ in $R$, see (\ref{2.1}). It is shown in \cite{KU2}  that $\grade J(\rho)\le t$  and if $\grade J(\rho)=t$, then $J(\rho)$ is a perfect ideal in $R$. It is also shown in \cite{KU2} that, once minor hypotheses are imposed (see the statement of Corollary~\ref{Main-C} for the details), then  every residual intersection of a grade three Gorenstein ideal is equal to $J(\rho)$ for some almost alternating matrix $\rho$. Furthermore, a resolution $\mathcal D^0(\rho)$ of $R/J(\rho)$ by free $R$-modules is given in \cite{KU2}; this resolution is minimal whenever the data permits such a claim.

Assume that $R$ is a standard graded ring and that the entries of
$\rho$ are linear forms from $R$. In this paper we give explicit formulas for  the Hilbert series and multiplicity of  $\bar R=R/J(\rho)$. Recall that the Hilbert series of the graded ring $S=\bigoplus_{0\le i}S_i$ is the formal power series  
$$\textstyle\HS_S(z)=\sum_i \lambda_{S_0}(S_i) z^i,$$
where $\lambda_{S_0}(\underline{\phantom{x}})$ represents the length of an $S_0$-module, and the multiplicity of $S$
is $$e_S=(\dim S)!
\lim_{i\to\infty}\frac{\lambda_S(S/\mathfrak m^i S)}{i^{\dim S}},$$
where $\mathfrak m$ is the maximal homogeneous ideal of $S$ and  ``$\dim$'' represents  Krull dimension.

Theorem~\ref{Main-T} is the main result of the paper. (The ideal $J(\rho)$ is defined in (\ref{2.1}).)
\begin{theorem}\label{Main-T} Let $R_0$ be an Artinian local ring, $R=\bigoplus_{0\le i} R_i$ be a standard graded $R_0$ algebra, $n$ and $t$ be positive integers, $\rho$ be an $n\times (n+t)$ almost alternating matrix with homogeneous linear entries, $J$ be the ideal $J(\rho)$, and $\bar R$ be the quotient ring $R/J$. If ${t\le \grade J}$, then the following statements hold.
\begin{enumerate}[\rm(a)]
\item\label{Main-T-a} The Hilbert series of $\bar R$ is
$\HS_{\bar R}(z)=
\HS_R(z)\cdot (1-z)^t \cdot \hn_{\bar R}(z)$, where
\begin{equation}\label{hnRbar}\hn_{\bar R}(z)=\sum\limits_{i=0}^{n-1} \binom{t+i-1}{i}z^i -
\sum\limits_{i=\lfloor\frac{n+1}2\rfloor}^{n-1} \binom{t+2i-n-1}{2i-n}z^{i}.\end{equation}
\item\label{Main-T-b} The multiplicity, $e_{\bar R}$, of $\bar R$ is equal to  
$$e_R\cdot\sum_{i=0}^{\lfloor\frac {n-1}2\rfloor}\binom{n-2-2i+t}{t-1},$$which is also equal to
$e_R$ times the number of monomials $m$ of degree at most $n-1$ in $t$ variables with $\deg m+n$ odd.

\item\label{Main-T-c} In particular, if $R$ is a standard graded polynomial ring over a field, then the $h$-vector of $\bar R$ is the vector of coefficients of the polynomial
$\hn_{\bar R}(z)$ and the multiplicity of $\bar R$ is $$\sum_{i=0}^{\lfloor\frac {n-1}2\rfloor}\binom{n-2-2i+t}{t-1}.$$
\end{enumerate}\end{theorem}

We highlight the application of Theorem~\ref{Main-T} to residual intersections; see \ref{ri} for a definition of residual intersection.

\begin{corollary}\label{Main-C} Let $R_0$ be an Artinian local ring, $R=\bigoplus_{0\le i} R_i$ be a Cohen-Macaulay standard graded $R_0$-algebra, $I$ be a linearly presented grade three Gorenstein ideal in $R$ {\rm(}presented by an alternating matrix{\rm)}, $A$ be a sub-ideal of $I$ minimally generated generated by $t$ homogeneous elements for some $t$ with $3\le t$, $J$ be the ideal $A:_RI$, and $\bar R=R/J$.
Assume that  \begin{enumerate}[\rm(1)]
\item the homogeneous minimal generators of $A$ live in two degrees: $\init(I)$ and $\init(I)+1$, where $\init(I)$ is the least degree $r$ for which $I$ contains a non-zero element of degree $r$,
\item the minimal number of generators of $I/A$ is $n$ for some positive integer $n$, and 
\item the ideal $J$ is a $t$-residual intersection of $I$ {\rm(}that is, $t\le \grade J${\rm)}.
\end{enumerate}
If either
\begin{enumerate}[\rm(i)]
\item the ring $R$ is Gorenstein, or else,
\item the residual intersection $J=A:_RI$ is geometric,
\end{enumerate} then conclusions {\rm (\ref{Main-T-a}), (\ref{Main-T-b})}, and  {\rm(\ref{Main-T-c})} of the Theorem~{\rm \ref{Main-T}} give the Hilbert series and multiplicity of the ring $\bar R$.
\end{corollary}

\begin{Remark} The multiplicity calculation of Corollary~\ref{Main-C} is carried out for at least some residual intersections in \cite{KPU-BA}. Both calculations begin with the resolutions of \cite{KU2} and both calculations involve some manipulation of binomial coefficients. The present calculation gives the entire $h$-vector of $\bar R$ in addition to the multiplicity (which is the sum of the entries in the $h$-vector).\end{Remark}

\begin{chunk}\label{ri}Let $R$ be a commutative Noetherian ring,
$I$
an
ideal
in $R$, $t$ an integer with $\operatorname{ht} I\le t$, $A$ a proper subideal of $I$ 
which can be generated by $t$ elements, and $J$ the ideal $A :_R I$.  If 
$t\le \operatorname{ht} J$, then $J$ is called an $t$-{\it residual intersection} of $I$.  If, 
furthermore, $I_P = A_P$ for all prime ideals $P$ of $R$ with $I \subseteq P$ 
and $\operatorname{ht} P \leq t$, then $J$ is called a {\it geometric $t$-residual 
intersection} of $I$.
\end{chunk} 

\begin{chunk}\label{2.1}Let $\rho = [X\ Y]$  be an almost 
alternating matrix, with $X$ a square matrix. The alternating matrix which corresponds to $\rho$ is
$$ T = \left[ \begin{matrix}
X & Y \\ -Y^{\rm t} & 0 \end{matrix} \right] . $$
Let $J(\rho)$ be the ideal which is generated by the Pfaffians of all principal 
submatrices of $T$ which contain $X$.\end{chunk}

\section{The  proof of Theorem~\ref{Main-T} and Corollary~\ref{Main-C}.}

We begin by giving a bi-graded version of the complexes $\mathcal D^0(\rho)$ from \cite[Sect.2]{KU2}. The following result is not stated in \cite{KU2}; but it could be. We actually make no use of the bi-homogeneous version of $\mathcal D^0(\rho)$; on the other hand, no extra work is involved in calculating the bi-graded twists rather than only the graded twists.
\begin{lemma}\label{11bi-gr} Let $B$ be a bi-graded Noetherian ring, and $n$ and $t$ be positive integers, and 
  $\rho=\bmatrix X&Y\endbmatrix $ be an almost alternating $n\times (n+t)$ with $X$ an alternating matrix. Assume that each entry of $X$ is bi-homogeneous of degree $(-1,0)$ and each entry of $Y$ is bi-homogeneous of degree $(0,-1)$. 
 Then the complex 
 $\mathcal D^0(\rho)$
is $$0\to \mathcal D_t\to\mathcal D_{t-1}\to \dots \to \mathcal D_1\to \mathcal D_0,$$ 
with$$\mathcal D_N=\begin{cases}
B,&\text{if $N=0$},\vspace{5pt}\\
\bigoplus\limits_{i=0}^{\lfloor\frac n2\rfloor}B(-i,2i-n)^{\binom t{n-2i}},\vspace{5pt} 
&\text{if $N=1$},\\
\bigoplus\limits_{i=0}^{n-1}
B(-i,i+1-N-n)^{\binom{N+n-2}{n-i-1}\binom{i+N-2}i\binom{t}{N+n-1-i}}
,&\text{if $2\le N\le t$, and}\vspace{5pt}\\
0,&\text{otherwise}.\end{cases}$$
\end{lemma}

\begin{corollary}If the hypotheses of the Theorem~{\rm\ref{Main-T}} are in effect, then 
\begin{equation}\label{1.2.1}\HS_{\bar R}(z)=\HS_{R}(z)\HN_{\bar R}(z),\end{equation}for
 \begin{equation}\label{HNRbar} \textstyle\HN_{\bar R}(z)= 1-\sum\limits_{I=\lfloor \frac{n+1}2\rfloor}^{n}
\binom t{2I-n}z^{I}+\sum\limits_{I=n+1}^{n+t-1}(-1)^{I-n+1}\sum\limits_{i=0}^{n-1} \binom{I-1}{n-i-1}\binom{i+I-n-1}i\binom{t}{I-i}z^{I}.\end{equation}
\end{corollary}
\begin{proof}
The hypotheses of Theorem~\ref{Main-T}, together with \cite[Thm.~8.3]{KU2}, guarantee that the complex $\mathcal D^0(\rho)$  is a resolution of $\bar R$. One now reads from Lemma~\ref{11bi-gr} that (\ref{1.2.1}) holds with
$$ \textstyle\HN_{\bar R}(z)= 1-\sum\limits_{i=0}^{\lfloor\frac {n}2\rfloor}
\binom t{n-2i}z^{n-i}+\sum\limits_{N=2}^t(-1)^N\sum\limits_{i=0}^{n-1} \binom{N+n-2}{n-i-1}\binom{i+N-2}i\binom{t}{N+n-1-i}z^{N+n-1}.$$
Let $I=n-i$ in the first sum and $I=N+n-1$ in the second sum to obtain the stated formulation.
\end{proof}

\begin{proof1}
We prove  (\ref{Main-T-a}) by  showing that 
$\HN_{\bar R}(z)=(1-z)^t\hn_{\bar R}(z)$,
where $\HN_{\bar R}(z)$ and $\hn_{\bar R}(z)$ are the polynomials given in 
(\ref{HNRbar}) and (\ref{hnRbar}), respectively.
An easy calculation yields that $(1-z)^t\hn_{\bar R}(z)$  is equal to
$$\textstyle
-\sum\limits_{I=\lfloor\frac{n+1}2\rfloor}^{n+t-1}\Big(
\sum\limits_{i=\lfloor\frac{n+1}2\rfloor}^{n-1}(-1)^{I-i} \binom{t+2i-n-1}{2i-n}\binom{t}{I-i}\Big)z^{I}+\sum\limits_{I=0}^{n+t-1}\Big(\sum\limits_{i=0}^{n-1}(-1)^{I-i}\binom{t+i-1}i\binom t{I-i}\Big)z^I
;$$
consequently, it suffices to prove
\begin{equation}\label{ShowMe1}\textstyle 1-\sum\limits_{I=\lfloor \frac{n+1}2\rfloor}^{n}
\binom t{2I-n}z^{I} =  \begin{cases}
-\sum\limits_{I=\lfloor\frac{n+1}2\rfloor}^{n}\Big(
\sum\limits_{i=\lfloor\frac{n+1}2\rfloor}^{n-1}(-1)^{I-i} \binom{t+2i-n-1}{2i-n}\binom{t}{I-i}\Big)z^I\\
+\sum\limits_{I=0}^{n}\Big(\sum\limits_{i=0}^{n-1}(-1)^{I-i}\binom{t+i-1}i\binom t{I-i}\Big)z^{I}\end{cases}\end{equation}
and
\begin{equation}\label{ShowMe2}\textstyle \begin{array}{ll}&\sum\limits_{I=n+1}^{n+t-1}(-1)^{I-n+1}\sum\limits_{i=0}^{n-1} \binom{I-1}{n-i-1}\binom{i+I-n-1}i\binom{t}{I-i}z^{I}\vspace{5pt}\\=&\sum\limits_{I=n+1}^{n+t-1}\Big(-
\sum\limits_{i=\lfloor\frac{n+1}2\rfloor}^{n-1}(-1)^{I-i} \binom{t+2i-n-1}{2i-n}
\binom{t}{I-i} +\sum\limits_{i=0}^{n-1}(-1)^{I-i}\binom{t+i-1}i\binom t{I-i}
\Big)z^{I}.\end{array}\end{equation}
The equation (\ref{ShowMe1}) is established in Lemma~\ref{proveSM1}. Remove the variable $z$ from (\ref{ShowMe2}), multiply by $(-1)^{I+n+1}$, and remove the unnecessary constraints on the index $i$. (Indeed, $\binom ab$ is zero if $b$ is negative.) To establish (\ref{ShowMe2}), it suffices to show that if 
$1\le I$, then

\begin{equation}\label{ShowMe2'}
\textstyle 0=\begin{cases}-\sum\limits_{i\in \Z} \binom{I-1}{n-i-1}\binom{i+I-n-1}i\binom{t}{I-i}\vspace{5pt}\\+
\sum\limits_{i\le n-1}(-1)^{n-i} \binom{t+2i-n-1}{2i-n}
\binom{t}{I-i}\\ +\sum\limits_{i\le n-1}(-1)^{n+1-i}\binom{t+i-1}i\binom t{I-i}
.\end{cases}\end{equation}
The right side of (\ref{ShowMe2'}) is called $Q(n-1,t,I,0)$ in Definition~\ref{3.1}. It is shown in Proposition~\ref{This is it.} that $0=Q(n-1,t,I,0)$. The hypotheses ($0\le t$, $0\le n-1$, and  $1\le I$) of Proposition~\ref{This is it.} are satisfied by the present data.

To prove (\ref{Main-T-b}), it suffices to calculate
\begingroup\allowdisplaybreaks
\begin{align}
\hn_{\bar R}(1)&\textstyle=\sum\limits_{i=0}^{n-1} \binom{t+i-1}{i} -
\sum\limits_{i=\lfloor\frac{n+1}2\rfloor}^{n-1} \binom{t+2i-n-1}{2i-n}
\notag\\
&\textstyle=\binom{t+n-1}{n-1} -
\sum\limits_{i\le n-1} \binom{t+2i-n-1}{2i-n}.
\notag\\\intertext{Let $j=2i-n$ 
to obtain}\hn_{\bar R}(1)&\textstyle=\binom{t+n-1}{n-1} -
\sum\limits_{\gf{j\le n-2}{\text{$j+n$ even}}} \binom{t+j-1}{j}.
\label{sum}\end{align}\endgroup
The binomial coefficient $\binom{t+n-1}{n-1}$ is equal to the number of monomials of degree at most $n-1$ in $t$ variables and the sum on the right side of (\ref{sum})  
is the number of monomials $m$ of degree at most $n-2$ in $t$ variables with $\deg m+n$ even. The difference is the number of monomials $m$ of degree at most $n-1$ in $t$ variables with $\deg m+n$ odd.

Assertion (\ref{Main-T-c}) requires no further proof.
\qed\end{proof1}

\begin{lemma}\label{proveSM1} If $n$ and $t$ are positive integers, then {\rm(\ref{ShowMe1})} holds. 
\end{lemma}\begin{proof}
The binomial coefficient $\binom t{I-i}$ is zero unless $i\le I$; consequently, if the upper limit for $i$ on the right side of (\ref{ShowMe1}) is changed from $n-1$ to $n$, the only value of $I$ which is affected is $I=n$, and, if $I=i=n$, then $\binom{t+i-1}i=\binom{t+2i-n-1}{2i-n}$. It suffices to show
\begin{equation}\label{240}\textstyle 1-\sum\limits_{I=\lfloor \frac{n+1}2\rfloor}^{n}
\binom t{2I-n}z^{I} =  \begin{cases}\phantom{+}\sum\limits_{I=0}^{n}\Big(\sum\limits_{i=0}^{n}(-1)^{I-i}\binom t{I-i}\binom{t+i-1}i\Big)z^I\\
-\sum\limits_{I=\lfloor\frac{n+1}2\rfloor}^{n}\Big(
\sum\limits_{i=\lfloor\frac{n+1}2\rfloor}^{n}(-1)^{I-i}\binom{t}{I-i} \binom{t+2i-n-1}{2i-n}
\Big)z^{I}.\end{cases}\end{equation}Observe next that 
\begin{equation}\label{ObserveNext}1=\sum\limits_{I=0}^{n}\Big(\sum\limits_{i=0}^{n}(-1)^{I-i}\binom t{I-i}\binom{t+i-1}i\Big)z^I.\end{equation}Indeed, the right side of (\ref{ObserveNext}) is equal to the first $n+1$ terms of the power series expansion of
$$1=(1-z)^t\frac{1}{(1-z)^t}.$$ Subtract (\ref{ObserveNext}) from (\ref{240}); multiply by $-1$; and look at one coefficient at a time.
Fix $I$ with $\lfloor \frac{n+1}2\rfloor\le I\le n$.
It suffices to show that \begin{equation}\label{progress}
\textstyle \binom t{2I-n} =  
\sum\limits_{i=\lfloor\frac{n+1}2\rfloor}^{n}(-1)^{I-i}\binom{t}{I-i} \binom{t+2i-n-1}{2i-n}.
\end{equation}
\noindent The right side of (\ref{progress}) is
\begingroup\allowdisplaybreaks\begin{align}
&\textstyle=\sum\limits_{i=\lfloor\frac{n+1}2\rfloor}^{I}(-1)^{I-i}\binom{t}{I-i} \binom{t+2i-n-1}{2i-n},\notag&&\textstyle \text{because $\binom{t}{I-i}=0$ if $I<i$, and $I\le n$,}\\
&\textstyle=\sum\limits_{j=0}^{\lfloor\frac{2I-n}2\rfloor}(-1)^{j}\binom{t}{j} \binom{2I-n-2j+ t-1}{2I-n-2j},\notag&&\text{with $j=I-i$.}
\end{align}\endgroup \noindent Thus, the right side of (\ref{progress}) is
\begingroup\allowdisplaybreaks\begin{align}
&\textstyle=\sum\limits_{j=0}^{\lfloor\frac{2I-n}2\rfloor}(\text{the coefficient of $z^{2j}$ in $(1-z^2)^t$})\cdot (\text{the coefficient of $z^{2I-n-2j}$ in $\frac 1{(1-z)^t}$})\notag\\
&\textstyle= \text{the coefficient of $z^{2I-n}$ in $\frac{(1-z^2)^t}{(1-z)^t}=(1+z)^t$}\notag\\
&\textstyle=\binom t{2I-n},\notag
\end{align}\endgroup
which is the left side of (\ref{progress}). \end{proof}

\begin{proof2}
The Corollary is an immediate consequence of  Theorem~\ref{Main-T} and \cite[10.2]{KU2}. We offer the following explanation of how the relevant $n\times (n+t)$ almost alternating matrix arises. Let $\mu$ be the minimal number of generators of $I$. One starts with a $\mu\times \mu$ matrix $X$ of linear forms which presents $I$ and a $\mu\times t$ matrix $Y$ which expresses the generators of $A$ in terms of the generators of $I$. One can arrange this data so that the matrix $Y$ has the form
$$Y=\bmatrix Y'&0\\0&I_{\mu-n}\endbmatrix,$$ where the entries of $Y'$ are linear forms and $I_{\mu-n}$ is the identity matrix with $\mu-n$ rows and columns.  The ideal $J$ is $J(\rho)$, where $\rho$ is the $\mu\times(\mu+t)$ almost alternating matrix $\bmatrix X&Y\endbmatrix$. The ideal $J$ is also $J(\rho')$, where $\rho'$ is the $n\times (n+t)$ almost alternating  matrix of linear forms which is obtained from $\rho$ by deleting the last $\mu-n$ rows and columns of $\rho$. \qed
\end{proof2}

\section{A family of identities}
\setcounter{equation}{0}
The calculations in this section are inspired by the Hilbert series calculations in \cite{KPU-ann}.

The identities \begin{equation}\label{crucial}Q(w,t,I,0)=0\end{equation}of Proposition~\ref{This is it.} are crucial to the proof of Theorem~\ref{Main-T}. The integers $Q(w,t,I,\a)$, with $1\le \a$, are introduced in order to prove (\ref{crucial}). We created $Q(*,t,*,\a+1)$ to be the first difference  
function $Q(*,t+1,*,\a)-Q(*,t,*,\a)$
of the the function $Q(*,t,*,\a)$. Fortunately, we are able to find a closed formula for $Q(*,t,*,\a+1)$ and thereby verify that this first difference  
function satisfies the desired initial condition $Q(*,0,*,\a+1)=0$.
\begin{definition}\label{3.1} For integers $w,t,I,\a$, with $0\le \a$, define $Q(w,t,I,\a)$ to be the integer
\begingroup\allowdisplaybreaks
\begin{align}
Q(w,t,I,0)&=\begin{cases}
-\sum\limits_{i\in\Z} \binom{I-1}{w-i}\binom{i+I-w-2}i\binom{t}{I-i}
\\
+
\sum\limits_{i\le w}(-1)^{i+w+1} \binom{2i-w-2+t}{2i-w-1}\binom t{I-i}
\\
+\sum\limits_{i\le w}(-1)^{i+w} \binom{t+i-1}{i}\binom t{I-i},
\end{cases}
&&\text
{if $\a=0$, and}\notag\\ 
Q(w,t,I,\a)&=
\begin{cases}
-\sum\limits_{i\in\Z} \binom{I-1}{w-i}\binom{i+I-w-2}i\binom{t}{I-i-\a}\\
+
\sum\limits_{i\le w}(-1)^{i+w+1+\a} \binom{2i-w-2+t+\a}{2i-w-1+\a}\binom {t}{I-i-\a}
\\
+\sum\limits_{i=1}^{\a-1}(-1)^{i+1}\binom{I-1}{w+i}\binom{t+w+i}{I-\a+i},
\end{cases}&&\text{
if $1\le \a$.}\notag\end{align}\endgroup\end{definition}

\begin{proposition}\label{This is it.} If  $w$, $t$, $I$, and $\a$ are integers with $\a$, $t$, $w+\a$, and $I-1$ all non-negative, then
$Q(w,t,I,\a)=0$.\end{proposition}
\begin{proof} Fix integers $w$, $I$, and $\a$ with $\a$, $w+\a$, and $I-1$ non-negative.
We show in Lemmas~\ref{L3.3} and \ref{L3.4}  that
\begin{align}
Q(w,t+1,I,\a)&=Q(w,t,I,\a)+Q(w,t,I,\a+1)\label{rr1}\\
Q(w,0,I,\a)&=0.\label{rr2}\end{align}
This recurrence relation now yields that $Q(w,t,I,\a)=0$ for all non-negative $t$. \end{proof}

\begin{lemma}\label{L3.3} If $w$, $t$, $I$, and $\a$ are integers  with $t$, $\a$, $w+\a$, and $I-1$ non-negative, then  equality holds in {\rm(\ref{rr1})}. 
\end{lemma}

\begin{proof}We treat the cases $\a=0$ and $1\le \a$ separately. We begin with 
$\a=0$ and we  
compute \begin{equation}\label{alpha=0}Q(w,t,I,1)-Q(w,t+1,I,0)+Q(w,t,I,0).\end{equation}Write $$Q(w,t,I,1)=S_1+S_2,\  -Q(w,t+1,I,0)=S_3+S_4+S_5,\text{ and }Q(w,t,I,0)=S_6+S_7+S_8,$$ for
\begingroup\allowdisplaybreaks\begin{align}
S_1&\textstyle = -\sum\limits_{i\in\Z} \binom{I-1}{w-i}\binom{i+I-w-2}i\binom{t}{I-i-1},\notag\\
S_2&\textstyle =\sum\limits_{i\le w}(-1)^{i+w} \binom{2i-w-1+t}{2i-w}\binom {t}{I-i-1}
,\notag\\
S_3&\textstyle =\sum\limits_{i\in\Z} \binom{I-1}{w-i}\binom{i+I-w-2}i\binom{t+1}{I-i}
,\notag\\
S_4&\textstyle =-
\sum\limits_{i\le w}(-1)^{i+w+1} \binom{2i-w-1+t}{2i-w-1}\binom {t+1}{I-i}
,\notag\\
S_5&\textstyle =-\sum\limits_{i\le w}(-1)^{i+w} \binom{t+i}{i}\binom {t+1}{I-i},\notag\\
S_6&\textstyle =-\sum\limits_{i\in\Z} \binom{I-1}{w-i}\binom{i+I-w-2}i\binom{t}{I-i}
,\notag\\
S_7&\textstyle =
\sum\limits_{i\le w}(-1)^{i+w+1} \binom{2i-w-2+t}{2i-w-1}\binom t{I-i}
,\text{ and}\notag\\
S_8&\textstyle =\sum\limits_{i\le w}(-1)^{i+w} \binom{t+i-1}{i}\binom t{I-i}.
\notag\end{align}\endgroup
Apply the Pascal identity
\begin{equation}\label{PI}\textstyle \binom ab= \binom {a-1}b+\binom {a-1}{b-1},\end{equation}which holds for all integers $a$ and $b$, to see that $S_1+S_3+S_6=0$. Apply the Pascal identity three times: first to write  $S_4=S_4'+S_4''$ with 
\begin{align}
S_4'&\textstyle =-
\sum\limits_{i\le w}(-1)^{i+w+1} \binom{2i-w-1+t}{2i-w-1}\binom {t}{I-i-1}
\text{ and}\notag\\
S_4''&\textstyle =-
\sum\limits_{i\le w}(-1)^{i+w+1} \binom{2i-w-1+t}{2i-w-1}\binom {t}{I-i}
,\notag
\end{align} 
\noindent and then to obtain
\begin{align}
S_2+S_4'&\textstyle =\sum\limits_{i\le w}(-1)^{i+w} \binom{2i-w+t}{2i-w}\binom {t}{I-i-1}
\text{ and}\notag\\
S_4''+S_7&\textstyle =
\sum\limits_{i\le w}(-1)^{i+w} \binom{2i-w-2+t}{2i-w-2}\binom {t}{I-i}.
\notag
\end{align}
\noindent Replace the index $i$ in $S_4''+S_7$ by $i+1$ and then combine with $S_2+S_4'$. The result is
$$\textstyle S_2+S_4+S_7=\binom{t+w}w\binom t{I-w-1}.$$
Write $S_5=S_5'+S_5''$ with
\begin{align}
S_5'&\textstyle =-\sum\limits_{i\le w}(-1)^{i+w} \binom{t+i}{i}\binom {t}{I-i-1},\notag\\
S_5''&\textstyle =-\sum\limits_{i\le w}(-1)^{i+w} \binom{t+i}{i}\binom {t}{I-i},\notag\notag\end{align}
Observe that $$\textstyle S_5''+S_8=-\sum\limits_{i\le w}(-1)^{i+w} \binom{t+i-1}{i-1}\binom t{I-i}.$$ 
Replace the index $i$ in $S_5''+S_8$ by $i+1$ and then combine with $S_5'$. The result is
$$\textstyle S_5+S_8=-\binom{t+w}{w}\binom{t}{I-w-1}.$$Thus, (\ref{alpha=0}), which is equal to $\sum_{i=1}^8S_i$, is also equal to $S_1+S_3+S_6$ plus $S_2+S_4+S_7$ plus $S_5+S_8$, and this is zero. The assertion holds when $\alpha =0$. In this part of the argument, we did not need to use the conditions imposed on $t$, $w$, and $I$.  

 The $1\le \a$ part of the argument is similar, but with important differences.
We compute \begin{equation}\label{1-le-alpha} Q(w,t,I,\a+1)-Q(w,t+1,I,\a)+Q(w,t,I,\a).\end{equation}
Write $Q(w,t,I,\a+1)=\sum_{i=1}^3S_1$, $-Q(w,t+1,I,\a)=\sum_{i=4}^6 S_i$ and 
$Q(w,t,I,\a)=\sum_{i=7}^9 S_i$, as before. Observe that $S_1+S_4+S_7=0$, 
\begin{align}\textstyle S_2+S_5+S_8&=\textstyle  (-1)^{\a}\binom{w+t+\a}{w+\a}\binom t{I-w-\a-1},\text{ and}\notag\\
\textstyle S_3+S_6+S_9&=\textstyle(-1)^{\a+1}\binom{I-1}{w+\a}\binom{t+w+\a}{I-1}.\notag\end{align}
Thus, (\ref{1-le-alpha}) is equal to
\begin{equation}\label{zero?}\textstyle (-1)^\alpha \Big[\binom{w+t+\a}{w+\a}\binom t{I-w-\a-1}-\binom{I-1}{w+\a}\binom{t+w+\a}{I-1}\Big].\end{equation}

\noindent The hypotheses guarantee that $w+t+\a$, $w+\a$, $t$, and $I-1$ all  are  non-negative. Furthermore, we observe that if $I-w-\a-1<0$ or $t<I-w-\a-1$, then both summands in (\ref{zero?}) are zero. Henceforth, we may assume $0\le I-w-\a-1$ or $I-w-\a-1\le t$. In this case, each binomial coefficient $\binom ab$ in (\ref{zero?}) satisfies $0\le b\le a$; consequently, $\binom ab=\frac{a!}{b!(a-b)!}$ and a straightforward calculation shows that 
(\ref{zero?}) is zero.
\end{proof}

\begin{lemma}\label{L3.4} If $w$, $I$, and $\a$ are integers  with $\a$, $w+\a$, and $I-1$ non-negative, then  equality holds in {\rm(\ref{rr2})}. 
\end{lemma}

\begin{proof} We evaluate $Q(w,0,I,\a)=S_1+S_2+S_3$, with $$
\begin{array}{lll}
S_1&=&-\sum\limits_{i\in\Z} \binom{I-1}{w-i}\binom{i+I-w-2}i\binom{0}{I-i-\a},\\

S_2&=&\sum\limits_{i\le w}(-1)^{i+w+1+\a} \binom{2i-w-2+\a}{2i-w-1+\a}\binom {0}{I-i-\a},\text{ and}
\\
S_3&=&\sum\limits_{i=1}^{\a-1}(-1)^{i+1}\binom{I-1}{w+i}\binom{w+i}{I-\a+i}.
\end{array} $$The binomial coefficient $\binom{0}{b}$ is zero unless $b=0$. If ``$S$'' is a statement, then we write $\chi(S)$ to mean
$$\chi(S)=\begin{cases} 1&\text{if ``$S$'' is true}\\0&\text{otherwise.}\end{cases}$$ We see that
\begin{align}\textstyle S_1&\textstyle= -\binom{I-1}{w-I+\a} \binom{2I-\a-w-2}{I-\a}&\text{ and}\label{S1}\\
\textstyle S_2&=\textstyle \chi(I-\a\le w)(-1)^{I+w+1}\binom{2I-\a-w-2}{2I-\a-w-1}.\notag\end{align}The binomial coefficient $\binom{b-1}{b}$ is zero unless $b=0$. Furthermore, if $w=2I-\a-1$, then $I-\a\le w$ holds automatically because $0\le I-1$. 
It follows that \begin{equation}\label{S2}S_2=(-1)^{I+w+1}\chi(w=2I-\a-1).\end{equation}

Let $k=w+i$ in $S_3$ in order to obtain
$$\textstyle S_3=\sum\limits_{k=w+1}^{w+\a-1}(-1)^{k-w+1}\binom{I-1}{k}\binom{k}{I-\a-w+k}=S_3'+S_3''+S_3''',$$with

 $$\begin{array}{lll}
S_3'&=&\sum\limits_{k\in \Z}(-1)^{k-w+1}\binom{I-1}{k}\binom{k}{I-\a-w+k}\\
\\
S_3''&=&\sum\limits_{k\le w}(-1)^{k-w}\binom{I-1}{k}\binom{k}{I-\a-w+k}\\ \\
S_3'''&=&\sum\limits_{w+\a\le k}(-1)^{k-w}\binom{I-1}{k}\binom{k}{I-\a-w+k}.
\end{array}$$
For $S_3'$, we use the  identity 
$$ \textstyle \sum\limits_{k\in\Z}(-1)^{k}\binom{b+k}{c+k}
\binom{a}{k}=(-1)^{a}\binom{b}{a+c},$$   
 which holds for integers $a$, $b$, and $c$ with    $0\le
a$. (See, for example, \cite[Lemma~1.3]{Ku92p}.)
 Take $a=I-1$, $b=0$, and $c=I-\a-w$. We conclude that \begin{equation}\label{3'}\textstyle S_3'=(-1)^{w+I}\binom 0{2I-1-\alpha-w}=(-1)^{w+I}\chi(w=2I-1-\a)=-S_2,\end{equation}for $S_2$ given in (\ref{S2}).
In $S_3'''$, the ambient hypothesis guarantees that $0\le I-1$. If the term corresponding to $k$ is non-zero, then 
$$0\le k\le I-1<I\le I-\a-w+k\le k.$$ The inequality $k<k$ never occurs; consequently, \begin{equation}\label{3'''}S_3'''=0.\end{equation}

Notice that $S_3''$ is zero unless $0\le I-\a$. So,
$$\begin{array}{lll}\textstyle S_3''&=&\chi(0\le I-\a) \sum\limits_{k\le w}(-1)^{k-w}\binom{I-1}{k}\binom{k}{I-\a-w+k}\\&=&\chi(0\le I-\a)\sum\limits_{k=\max\{0,w+\a-I\}}^{\min\{w,I-1\}}(-1)^{k-w}\binom{I-1}{k}\binom{k}{I-\a-w+k}.\end{array}$$
We next show that \begin{equation}\label{the proposed identity}\textstyle S_3''=\chi(0\le I-\a)\binom{I-1}{w+\a-I}\sum\limits_{k=\max\{0,w+\a-I\}}^{\min\{w,I-1\}}(-1)^{k-w}\binom{2I-1-\a-w}{I-1-k}.\end{equation} Observe  that
 if $w+\a-I<0$, then both sides of (\ref{the proposed identity}) are zero; and therefore, in order to establish (\ref{the proposed identity}), it suffices to prove that
\begin{equation}\label{concl}\textstyle \binom{I-1}{k}\binom{k}{I-\a-w+k}=\binom{I-1}{w+\a-I}\binom{2I-1-\a-w}{I-1-k}\end{equation} when 
\begin{equation}\label{hyp}\max\{0,w+\a-I\}\le k\le \min\{w,I-1\}\quad \text{and}\quad 0\le w+\a-I.\end{equation}
The hypotheses (\ref{hyp}) guarantee that each of the four binomial coefficients $\binom ab$ from (\ref{concl}) satisfies $0\le a\le b$; consequently each of these binomial coefficients is equal to $\frac{a!}{b!(a-b)!}$. At this point (\ref{concl}) can be established with no difficulty; and therefore, the equation (\ref{the proposed identity}) has been established.

Let $\ell=I-1-k$ in (\ref{the proposed identity}) to see  that 
$$\textstyle S_3''=\chi(0\le I-\a)\binom{I-1}{w+\a-I}\sum\limits_{\ell=\max\{0,I-1-w\}}^{\min\{2I-1-w-\a,I-1\}}(-1)^{I-1-w-\ell}\binom{2I-1-\a-w}{\ell}$$
The constraint $0\le \ell$ is not needed because $\binom{2I-1-\a-w}{\ell}=0$ when $\ell<0$. On the other hand, as we think about $S_3''$ we may as well assume $0\le w+\a-I$ (otherwise $S_3''=0$). It follows that $I-w-\a\le0$ and $2I-1-w-\a\le I-1$. Thus,
\begingroup\allowdisplaybreaks\begin{align}\textstyle S_3''&\textstyle =\chi(0\le I-\a)\binom{I-1}{w+\a-I}\sum\limits_{\ell=I-1-w}^{2I-1-w-\a}(-1)^{I-1-w-\ell}\binom{2I-1-\a-w}{\ell}=T_1+T_2,\notag\\\intertext{for}
 T_1&\textstyle=\chi(0\le I-\a)\binom{I-1}{w+\a-I}(-1)^{I-1-w}
\sum\limits_{\ell\le 2I-1-w-\a}
(-1)^{\ell}\binom{2I-1-\a-w}{\ell}\notag\\\intertext{and}
T_2&\textstyle=-\chi(0\le I-\a)\binom{I-1}{w+\a-I}(-1)^{I-1-w}
\sum\limits_{\ell\le I-2-w}(-1)^\ell\binom{2I-1-\a-w}{\ell}.\notag
\end{align}\endgroup

(The factor $\chi(0\le I-\a)$ is very important. If $I-\a$ is less than zero, then $S_3''$ is zero, but
$$\textstyle\binom{I-1}{w+\a-I}(-1)^{I-1-w}\Big[\sum\limits_{\ell\le 2I-1-w-\a}(-1)^{\ell}\binom{2I-1-\a-w}{\ell}
-\sum\limits_{\ell\le I-2-w}(-1)^\ell\binom{2I-1-\a-w}{\ell}\Big]$$ is not necessarily zero.)
Apply Lemma~\ref{little lemma} to see that 
$$\begin{array}{lll} T_1&=&
\chi(0\le I-\a)\binom{I-1}{w+\a-I}
(-1)^{I-\a}\binom{2I-w-\a-2}{2I-w-\a-1}\text{ and}\\ \\
T_2&=&
\chi(0\le I-\a)\binom{I-1}{w+\a-I}
\binom{2I-2-\a-w}{I-2-w}.
\end{array}$$

The binomial coefficient $\binom {b-1}b$ is zero unless $b=0$ and $\binom{I-1}{I-1}=1$, since $0\le I-1$; therefore,

\begin{equation}\label{S3''a}T_1=(-1)^{I-\a}\chi(0\le I-\a)
\chi(w=2I-1-\a).\end{equation}

The integer $T_2$ is zero unless $0\le I-2-w$. If $0\le I-2-w$, then
$$0\le (I-\a)+(I-2-w)=2I-\a-2-w$$ and $\binom{2I-2-\a-w}{I-2-w}=\binom{2I-2-\a-w}{I-\a}$. Thus,
$$\textstyle  T_2=
\chi(0\le I-\a)\chi(0\le I-2-w)\binom{I-1}{w+\a-I}
\binom{2I-2-\a-w}{I-\a}.
$$ The factor  $\binom{2I-2-\a-w}{I-\a}$ of $T_2$ makes the factor $\chi(0\le I-\a)$ redundant; hence,
\begingroup\allowdisplaybreaks\begin{align}
T_2&=\textstyle \phantom{-}\chi(0\le I-2-w)\binom{I-1}{w+\a-I}
\binom{2I-2-\a-w}{I-\a}.\notag\\ \intertext{Recall the integer $S_1$ from (\ref{S1}) and observe that}
S_1+T_2&=\textstyle-\chi(I-1\le w) \binom{I-1}{w+\a-I}\binom{2I-2-\a-w}{I-\a}.\notag\\\intertext{Apply the identity $\binom ab=(-1)^b\binom {b-a-1}b$, which holds for all integers $a$ and $b$, to write}
 S_1+T_2&=\textstyle -(-1)^{I-\a}\chi(I-1\le w) \binom{I-1}{w+\a-I}\binom{w+1-I}{I-\a}.\notag
\end{align}\endgroup
 If $S_1+T_2$ is non-zero, then  $0\le I-1$ and $0\le w+1-I$. Thus, if $S_1+T_2$ is non-zero, then 
$$w+\a-I\le I-1\quad\text{and}\quad I-\a\le w+1-I;$$ 
$$w+\a-I\le I-1\le w+\a-I;$$ and $w=2I-\a-1$. We have shown that
\begingroup\allowdisplaybreaks \begin{align}S_1+T_2&=\textstyle -(-1)^{I-\a}\chi(w=2I-\a-1)\chi(I-1\le w) \binom{I-1}{w+\a-I}\binom{w+1-I}{I-\a}\notag\\
&\textstyle =-(-1)^{I-\a}\chi(w=2I-\a-1)\chi(\a\le I) \binom{I-1}{I-1}\binom{I-\a}{I-\a}\notag\\
&=-(-1)^{I-\a}\chi(w=2I-\a-1)\chi(\a\le I) .\notag\\\intertext{Recall the value of $T_1$ from (\ref{S3''a}). We have shown that}
 S_1+T_2&=-T_1.\label{the last three} 
\end{align}\endgroup
Combine (\ref{the last three}), (\ref{3'}), and (\ref{3'''}) to see that
 $Q(w,0,I,\a)$,
which is equal to
$$(S_1+T_1+T_2)+(S_2+S_3')+S_3''' ,$$ is zero. 
\end{proof}

\begin{lemma}\label{little lemma} If $a$ and $b$ are integers, then
$$\sum\limits_{\ell\le b}(-1)^{\ell}\binom a\ell=(-1)^b\binom{a-1}b.$$
\end{lemma}\begin{proof}If $b<0$, then both sides are zero. If $b=0$, then both sides are $1$. A short induction completes the argument for $0<b$. The value of $a$ is not relevant; Pascal's identity (\ref{PI}) holds for all integers.\end{proof}

\end{document}